\documentclass[11pt]{amsart}

\usepackage{latexsym,rawfonts}
\usepackage{amsfonts,amssymb}
\usepackage{amsmath,amsthm}

\usepackage[plainpages=false]{hyperref}
\usepackage{graphicx}
\usepackage{color}
\usepackage[table]{xcolor}
\usepackage{longtable}

\newcommand{\rn}{\mathbb R}
\newcommand{\rnnnn}{\mathbb R^3}
\newcommand{\rnnn}{\mathbb R^{n+1}}

\newcommand{\sn}{ {\mathbb{S}^{n}}}

\newcommand{\N}{\mathbb N}

\newcommand{\psum}{{+_{\negthinspace\kern-2pt p}}\,}
\newcommand{\qsum}[1]{{+_{\negthinspace\kern-2pt #1}}\,}
\newcommand{\dpsum}{{\tilde+_{\negthinspace\kern-1pt p}}\,}
\newcommand{\dqsum}[1]{{\tilde+_{\negthinspace\kern-1pt #1}}\,}
\newcommand{\lsub}[1]{\hskip -1.5pt\lower.5ex\hbox{$_{#1}$}}

\numberwithin{equation}{section}

\newtheorem{theo}{Theorem}[section]

\newtheorem{lem}[theo]{Lemma}
\newtheorem{prop}[theo]{Proposition}

 \theoremstyle{definition}

\begin{document}

\title{The dual Minkowski problem for positive indices
}
\author[J. Hu]{Jinrong Hu}
\address{Institut f\"{u}r Diskrete Mathematik und Geometrie, Technische Universit\"{a}t Wien, Wiedner Hauptstrasse 8-10, 1040 Wien, Austria
 }
\email{jinrong.hu@tuwien.ac.at}

\begin{abstract}
We derive the stability result of the dual curvature measure with near-constant density in the even case. As an application, the existence and uniqueness of solutions to the even dual Minkowski problem for positive indices in $\rnnn$  are obtained with $n\geq 1$, provided the density of the given measure is close to 1 in the $C^{\alpha}$ norm with $\alpha\in (0,1)$.
\end{abstract}


\thanks{This work was supported by the Austrian Science Fund (FWF): Project P36545.}

\maketitle

\baselineskip18pt

\parskip3pt

\section{Introduction}

The geometric properties of convex bodies in Euclidean space $\rnnn$ and their corresponding Minkowski problems play a central role in the Brunn-Minkowski theory. The classical Minkowski problem, which characterizes the surface area measure, was originally formulated and studied by Minkowski himself in \cite{M897,M903}. Since then, significant progress has been made regarding the existence, uniqueness, and regularity of solutions to this problem in a series of works \cite{A39,A42,CY76,CW06,FJ38,HgLYZ05,N53}. Beyond area measures, the curvature measures introduced by Federer \cite{S14} form another fundamental class of measures within the Brunn-Minkowski theory.

The $L_{p}$ Brunn-Minkowski theory is a generalization of the Brunn-Minkowski theory, initiated by Firey and further developed by Lutwak \cite{L93,L96} through his introduction of the $L_{p}$ surface area measure. The $L_{p}$ Minkowski problem, which involves prescribing the $L_{p}$ surface area measure, is a fundamental problem that was first formulated and solved by Lutwak \cite{L93}. Building on Lutwak's foundational work, this problem has seen rapid development, as evidenced by numerous works, including \cite{B19,B17,CL17,F62,LW13,LYZ04,Zhu14,Zh15,Zhu15}. Among the most challenging cases are the logarithmic Minkowski problem for $p=0$ (see B\"{o}r\"{o}czky-Lutwak-Yang-Zhang \cite{BLYZ12} and its citations) and the centro-affine Minkowski problem for $p=-n-1$ (see Chou-Wang \cite{CW06} or Zhu \cite{Zhu15}, and their citations).

As another parallel extension of the Brunn-Minkowski theory, the dual Brunn-Minkowski theory was initiated by Lutwak in 1970s \cite{L75}. However, it truly gained significant momentum when Huang-Lutwak-Yang-Zhang \cite{HLYZ16} discovered a new family of fundamental geometric measures, known as the dual curvature measures. These measures are dual to Federer's curvature measures. Their work introduced the dual Minkowski problem, which concerns the prescription of the dual curvature measures, and further established sufficient conditions for the existence of even solutions in the case $0<q<n+1$ within the variational framework. Since then, the dual Brunn-Minkowski theory has flourished, leading to numerous significant results and applications, as explored in \cite{BF19,BLYZ19,BLYZ20,HLYZ18,HZ18,JW17,LSW20,Z17,Z18}  and the references therein.

In the work of Lutwak-Yang-Zhang \cite{LYZ18}, they first integrated the dual Brunn-Minkowski theory into the $L_{p}$ Brunn-Minkowski theory. Furthermore, they introduced a unifying family of geometric measures, referred to as the $(p,q)$-th dual curvature measures, which serve as fundamental geometric measures in the $L_{p}$ dual Brunn-Minkowski theory. The associated Minkowski problem, called the $L_{p}$ dual Minkowski problem, encompasses all the Minkowski-type problems mentioned above. The precise formulation is given below.

\emph{The $L_{p}$ dual Minkowski problem \cite{LYZ18}:} given a nonzero finite Borel measure $\mu$ on the unit sphere $\sn$ and real numbers $p,q$, what are the necessary and sufficient conditions for $\mu$ to coincide with $(p,q)$-th dual curvature measure of a convex body $K$ containing the origin in its interior?

When the given measure $\mu$ has a continuous density $f$, the solvability of  the $L_{p}$ dual Minkowski problem amounts to solving the following {M}onge-{A}mp\`ere equation on $\sn$:
\begin{equation}\label{Mong3}
h^{1-p}\det(\nabla^{2} h+hI)|Dh|^{q-(n+1)}=f, \quad {\rm on} \ \  \sn.
\end{equation}
In the case $p=0$, Eq. \eqref{Mong3} corresponds to the dual Minkowski problem:
 \begin{equation}\label{Mong}
h\det(\nabla^{2} h+hI)|Dh|^{q-(n+1)}=f, \quad {\rm on} \ \  \sn.
\end{equation}

 Since the publication of  \cite{LYZ18}, significant progress has been made in the study of the existence and uniqueness of solutions to \eqref{Mong3}, including important contributions from \cite{BF19,CHZ19,CL21,HZ18,LLJ22,LP21,Mu2}. In particular, when the density function $f$ is a constant, \eqref{Mong3} reduces to the isotropic $L_{p}$ dual Minkowski problem. The uniqueness and non-uniqueness results of solutions to the isotropic case have been widely investigated in the condition of symmetric or non-symmetric assumptions, for instance, in \cite{AN99,BCD17,CHZ19,CCL21,F74,HI24,HZ18,LW22,LW24}. Recently, in the plane case $n=1$, Kim-Lee \cite{KL24} established a uniform diameter estimate for solutions to \eqref{Mong3} in $\mathbb{R}^2$ when $0< p<1$ and $q\geq 2$, and moreover verified the uniqueness and positivity of solutions to the $L_{p}$ Minkowski problem when $f$ is sufficiently close to a constant in the $C^{\alpha}$ norm with $0<\alpha<1$. Separately, Chen-Li \cite{CL18} provided the diameter estimate for the case $p=0$. The diameter estimate  for the case $p>q$ can be readily derived using a maximum principle argument, as demonstrated in \cite{HZ18}.  However, the situation for $p\leq q$ is more complicated. It is interesting to establish the diameter estimate of solutions to \eqref{Mong3} with $p\leq q$ for higher dimensions $n\geq 2$.

 In this paper, our first purpose is to establish a stability result for the dual curvature measure with near-constant density in the even case utilizing the local Brunn-Minkowski inequality, inspired by the work of Hu-Ivaki \cite{HI25}. We need recall the following uniqueness result of solutions to the isotropic $L_{p}$ dual Minkowski problem shown by Ivaki-Milman \cite{IM23}.

\begin{theo}\cite{IM23}\label{ori}
Let $n\geq 1$. Suppose $p\geq -(n+1)$ and $q\leq n+1$ with at least one of these inequalities being strict. Let $\partial K$ be a smooth, strictly convex and origin-centred hypersurface with  the support function $h$ such that $h^{p-1}|Dh|^{n+1-q}\kappa=1$. Then $\partial K$ is the unit sphere.

\end{theo}

When $K$ is additionally assumed to be origin-symmetric, this uniqueness theorem was proved by Chen-Huang-Zhao \cite{CHZ19} from another point of view, provided $p\geq -(n+1)$ and $q\leq \min\{n+1, n+1+p\}$ with $p\neq q$. Now the stability result for the dual curvature measure is stated as follows.

\begin{theo}\label{main}
Let $n\geq 1$. Suppose $n-3\leq q \leq n+1$. Let $K$ be a smooth, origin-symmetric and strictly convex body with the support function $h$. Then
\[
\delta_{2}(\bar{K}, B_{1})\leq \beta\left[\frac{\max_{\sn} (|Dh|^{q-(n+1)}\frac{h}{\kappa})}{\min_{\sn} (|Dh|^{q-(n+1)}\frac{h}{\kappa})}-1  \right]^{\frac{1}{2}},
\]
where $\delta_{2}$ is the $L^{2}$-distance (see Section \ref{Sec2} for its definition),  $\beta$ is  a positive constant depending only on $n$, and
\[
\bar{K}=\frac{K}{\int_{\sn}hd\sigma/\int_{\sn}d\sigma}.
\]

\end{theo}

This theorem tells us that if the density of the dual curvature measure of a smooth, origin-symmetric and strictly convex body regarding to the spherical Lebesgue measure is approximately constant, a rescaled version of the body is close to the unit ball in the $L^{2}$-distance. Note that when $q=n+1$ in \eqref{Mong}, some progress on the stability of the cone-volume measure was made by B\"{o}r\"{o}czky-De \cite{BD21}, based on the log-Minkowski inequality in the class of convex bodies with many symmetries, as proved by B\"{o}r\"{o}czky-Kalantzopoulos \cite{BK22}. Subsequently, Ivaki \cite{IV22} proved the stability of the cone-volume measure in the class of origin-symmetric bodies with respect to the $L^{2}$-distance, the recent work of Hu-Ivaki \cite{HI25} proved a similar stability result without symmetry conditions.

An application of Theorem \ref{main} is to get the existence of even solutions to the regular dual Minkowski problem for positive indices  using degree theory methods, as well as its uniqueness, provided the prescribed data is sufficiently close to 1 in the $C^{\alpha}$ norm for $0<\alpha<1$.

\begin{theo}\label{coro2}
Suppose that $q$ satisfies either $ 0<q\leq n$ if $ \ 1\leq n\leq 3$,  or $ n-3\leq q\leq n$ if $ n>3$. Let $\alpha\in (0,1)$.  Let $f$ be an even, smooth and positive function on $\sn$. There exists a constant $\varepsilon_{0}>0$ depending only on $n$, $\alpha$ such that if  $||f-1||_{C^{\alpha}}\leq \varepsilon$ for some small $\varepsilon\in (0,\varepsilon_{0})$, then Eq. \eqref{Mong} has a unique, smooth, origin-symmetric and strictly convex solution.
\end{theo}

It should be stressed that the range of $q$ in Theorem \ref{coro2} guarantees both the existence and uniqueness of solutions to \eqref{Mong}. We also remark that, the existence of the even dual Minkowski problem in the smooth category for $q>0$ was previously demonstrated by Li-Sheng-Wang \cite{LSW20} from the perspective of geometric flows. The key ingredient of deriving the solvability of \eqref{Mong} is to obtain the $C^{0}$ estimate of solutions to \eqref{Mong}. In addition, partial uniqueness results for solutions to the (anisotropic) dual Minkowski problem \eqref{Mong} have been established.  For $q<0$,  Zhao \cite{Z17} proved uniqueness; the case $q=0$ is classical and stems from the uniqueness result of intergral curvature shown by Aleksandrov; for the logarithmic Minkowski problem ($q=n+1$), uniqueness results were found by  B\"{o}r\"{o}czky-Lutwak-Yang-Zhang \cite{BLYZ13} when the given measure is even in the planar case $n = 1$, by Chen-Huang-Li-Liu \cite{CY20} when the density $f$ of the given measure is even and close to 1 in the $C^{\alpha}$ norm, building on the local results given by Kolesnikov-Milman \cite{KM22}, and recently by Chen-Feng-Liu \cite{CFL22} when $f$ is close to 1 in the $C^{\alpha}$ norm without any symmetry condition in $\rnnnn$, later  B\"{o}r\"{o}czky-Saroglou \cite{BS24} and Hu-Ivaki \cite{HI25} extended  the results of \cite{CFL22} to higher dimensions independently, along different lines. However, the uniqueness in the case of $0<q\neq n+1$  has not been previously settled and remains largely open. Meanwhile, it would be interesting to generalize Theorem \ref{coro2} to the non-even case.

The structure of this paper is organized as follows. In Section \ref{Sec2}, we provide some basic facts related to convex bodies. In Section \ref{Sec3}, we present  the proof of Theorem \ref{main}. Finally, the proof of Theorem \ref{coro2} is given in Section \ref{Sec4}.

\section{Preliminaries}
\label{Sec2}

There are many standard references on the theory of convex bodies, including the comprehensive books of  Gardner \cite{G06} and Schneider \cite{S14}.

Let ${\rnnn}$ denote the $(n+1)$-dimensional Euclidean space.   For $Y, Z\in {\rnnn}$, $ \langle Y, Z\rangle $ represents the standard inner product. For a vector $X\in{\rnnn}$, $|X|=\sqrt{\langle X, X\rangle}$ is the Euclidean norm. Let $B_{1}$  denote the unit ball in $\rnnn$ and ${\sn}$ denote the unit sphere.  A convex body is defined as a compact convex set of ${\rnnn}$ with non-empty interior.

The support function of a convex body $K$ in $\rnnn$ (with respect to the origin) is defined for $x\in{\sn}$ as
\[
h_{K}(x)=\max\{\langle x, Y\rangle:Y \in K\}.
\]
Unless it causes confusion, we later abbreviate $h_{K}$ as $h$.

 The radial function $\rho_{K}$ of $K$ is denoted by
\begin{equation*}
\rho_{K}(u)=\max\{s>0: su\in K\}, \quad \forall u\in \sn.
\end{equation*}
Note that $\rho_{K}(u)u\in \partial K$ for any $u\in \sn$. Abbreviate $\rho_{K}$ as $\rho$ later unless it causes confusion.

 The $L^{2}$-distance of two convex bodies $K_{1}, K_{2}$ is expressed as
\[
\delta_{2}(K_{1},K_{2}):=\left( \frac{1}{\int_{\sn}d\sigma} \int_{\sn}|h_{K_{1}}-h_{K_{2}}|^{2}d\sigma\right)^{\frac{1}{2}}
\]
and their Hausdorff distance is defined as
\[
\delta_{H}(K_{1},K_{2}):=\max_{\sn}|h_{K_{1}}-h_{K_{2}}|.
\]

 Given a convex body $K$ in $\rnnn$, for $\mathcal{H}^{n}$ almost all $X\in \partial K$, the unit outer normal of $K$ at $X$ is unique. In this case, we denote by $\nu_{K}$ the Gauss map, which assigns to each $X\in \partial' K$ to its unique unit outer normal, where $\mathcal{H}^{n}(\partial K \backslash \partial' K)=0$ and $\partial' K$ is a Borel set (see, e.g. Schneider \cite{S14}). Therefore, the map $\nu_{K}$ is almost everywhere defined on $\partial K$.
Furthermore, for $\omega\subset {\sn}$, the inverse Gauss map $\nu_{K}$ is expressed as
\begin{equation*}
\nu^{-1}_{K}(\omega)=\{X\in \partial' K:  \nu_{K}(X) {\rm \ is \ defined \ and }\ \nu_{K}(X)\in \omega\}.
\end{equation*}

For a smooth and strictly convex body $K$, that is, a body whose boundary is $C^{\infty}$-smooth and is of positive Gauss curvature, we abbreviate $\nu^{-1}_{K}$ as $F$ for simplicity in the subsequent discussion. Then the support function of $K$ can be represented as
\begin{equation}\label{hhom}
h(x)=\langle x, F(x)\rangle=\langle\nu_{K}(X), X\rangle, \ {\rm where} \ x\in {\sn}, \ \nu_{K}(X)=x \ {\rm and} \ X\in \partial K.
\end{equation}
 Let $\{e_{1},e_{2},\ldots, e_{n}\}$ be a local orthonormal frame on ${\sn}$, and let $h_{i}$ denote the first order covariant derivatives of $h(\cdot)$  with respect to a local orthonormal frame on ${\sn}$. Differentiating \eqref{hhom} with respect to $e_{i}$ , we derive
\[
h_{i}=\langle e_{i}, F(x)\rangle+\langle x, F_{i}(x)\rangle.
\]
Since $F_{i}$ is tangent to $ \partial K$ at $F(x)$, there is
\begin{equation}\label{Fi}
h_{i}=\langle e_{i},F(x)\rangle.
\end{equation}
Combining \eqref{hhom} with \eqref{Fi}, we have
\begin{equation}\label{Fdef}
F(x)=\sum_{i} h_{i}e_{i}+hx=\nabla h+hx.
\end{equation}
Here $\nabla$ denotes the (standard) spherical gradient. On the other hand, since we can extend $h(\cdot)$ to $\rnnn$ as a 1-homogeneous function $h(\cdot)$, restricting the gradient of $h(\cdot)$ on $\sn$ yields
\begin{equation}\label{hf}
D h(x)=F(x), \ \forall x\in{\sn},
\end{equation}
where $D$ is the gradient operator in $\rnnn$. Let $h_{ij}$ be the second-order covariant derivatives of $h$ regarding a local orthonormal frame on ${\sn}$. Then, applying \eqref{Fdef} and \eqref{hf}, we have
\begin{equation*}\label{hgra}
D h(x)=\sum_{i}h_{i}e_{i}+hx, \quad F_{i}(x)=\sum_{j}(h_{ij}+h\delta_{ij})e_{j}.
\end{equation*}

Denote by $\sigma_{k}$ ($1\leq k \leq n$) the $k$-th elementary symmetric function of principal radii of curvature. The eigenvalues of matrix $\{h_{ij}+h\delta_{ij}\}$,  denoted by $\lambda=(\lambda_{1},\ldots, \lambda_{n})$, represent the principal radii of curvature at the point $X(x)\in \partial K$. Consequently, $\sigma_{1}=\Delta h+nh$, where $\Delta$ is the spherical Laplace operator, and the Gauss curvature $\kappa$  of $\partial K$ is given by
\begin{equation*}
\kappa=\frac{1}{\sigma_{n}}=\frac{1}{\det(h_{ij}+h\delta_{ij})}.
\end{equation*}

\section{Stability of the dual curvature measure}
\label{Sec3}
The following lemma is the local version of Aleksandrov-Fenchel inequality and also represents its spectral formulation, which originates from Hilbert's work (see, e.g. \cite{AN97, AB20}) and has been further studied, for example,  in \cite{IM23, KM22, ML22}.

\begin{lem}\label{SF}\cite{AN97, AB20}
Let $f\in C^{2}(\sn)$ with $\int_{\sn}fh\sigma_{k}d\sigma=0$. Then we get
\begin{equation*}
k\int_{\sn}f^{2}h\sigma_{k}d\sigma\leq \int_{\sn}\sum_{i,j}h^{2}\sigma^{ij}_{k}\nabla_{i}f\nabla_{j}fd\sigma,
\end{equation*}
where $\sigma^{ij}_{k}=\frac{\partial \sigma_{k}}{\partial b_{ij}}$ with $b_{ij}:=h_{ij}+h\delta_{ij}$. Equality holds if and only if for some vector $v\in \rnnn$, we have
\[
f(x)=\langle \frac{x}{h(x)},v\rangle,  \quad \forall x\in \sn.
\]
\end{lem}
By virtue of Lemma \ref{SF} for $k=n$, we obtain the following result, see also \cite[Lemma 2.15]{LW24}.

\begin{lem}\label{IN}
Let $X=Dh:\sn\rightarrow \partial K$ and $\alpha\in \rn$. Then we have
\begin{equation*}
\begin{split}
n\int_{\sn}|X|^{\alpha+2}dV_{n}&\leq n\frac{|\int_{\sn}|X|^{\frac{\alpha}{2}}X dV_{n}|^{2}}{\int_{\sn}dV_{n}}+\int_{\sn}|X|^{\alpha}h(\Delta h+nh)dV_{n}\\
&\quad +\left(\frac{\alpha^{2}}{4}+\alpha\right)\int_{\sn}|X|^{\alpha-1}h \langle \nabla h, \nabla |X|\rangle d V_{n},
\end{split}
\end{equation*}
where $dV_{n}=h\sigma_{n}d\sigma$.
\end{lem}
\begin{proof}

 Let  $\{e_{i}\}^{n}_{i=1}$ be a local orthonormal frame of $\sn$ so that $(h_{ij}+h\delta_{ij})(x_{0})=\lambda_{i}(x_{0})\delta_{ij}$. Suppose $\{E_{l}\}^{n+1}_{l=1}$ is an orthonormal basis of $\rnnn$. Motivated by \cite[Lemma 3.2]{IM23}, for $l=1,\ldots,n+1$, we set
the functional $f_{l}:\sn\rightarrow \rn$ as
\begin{equation*}
f_{l}(x)=|X|^{\frac{\alpha}{2}}\langle X(x),E_{l}\rangle-\frac{\int_{\sn}|X|^{\frac{\alpha}{2}}\langle X(x),E_{l}\rangle dV_{n}}{\int_{\sn}dV_{n}}.
\end{equation*}
It is clear to see  $\int_{\sn}f_{l}dV_{n}=0$ for $1\leq l \leq n+1$, then by means of Lemma \ref{SF} to $f_{l}$ and summing over $l$, there is
\begin{equation}\label{dx}
n\sum_{l}\int_{\sn}f^{2}_{l}dV_{n}=n\left[ \int_{\sn}|X|^{\alpha+2}dV_{n}-\frac{\Big|\int_{\sn}|X|^{\frac{\alpha}{2}}X dV_{n}\Big|^{2}}{\int_{\sn}dV_{n}} \right]\leq \sum_{l,i,j}\int_{\sn}h^{2}\sigma^{ij}_{n}\nabla_{i}f_{l}\nabla_{j}f_{l}d\sigma.
\end{equation}
Due to $\nabla_{i}X=\sum_{j}(h_{ij}+h\delta_{ij})e_{j}=\lambda_{i}e_{i}$ at $x_{0}$, we have $\langle e_{i}, X\rangle=h_{i}$ and $\sum_{i}\lambda_{i}\langle e_{i},X\rangle^{2}=|X|\langle \nabla h, \nabla |X|\rangle$ at $x_{0}$. Employing $\sum_{i}\frac{\partial \sigma_{n}}{\partial \lambda_{i}}\lambda^{2}_{i}=\sigma_{1}\sigma_{n}$ and $\frac{\partial \sigma_{n}}{\partial \lambda_{i}}\lambda_{i}=\sigma_{n}$ for $\forall i$, there holds
\begin{equation}
\begin{split}
\label{cwq}
\sum_{l,i,j}\sigma^{ij}_{n}\nabla_{i}f_{l}\nabla_{j}f_{l}&=\sum_{l,i}\frac{\partial \sigma_{n}}{\partial \lambda_{i}}((\nabla_{i}(|X|^{\frac{\alpha}{2}}))\langle X, E_{l}\rangle+|X|^{\frac{\alpha}{2}}\langle \nabla_{i}X,E_{l}\rangle)^{2}\\
&=\sum_{l,i}\frac{\partial \sigma_{n}}{\partial \lambda_{i}}\left(\frac{\alpha}{2} |X|^{\frac{\alpha}{2}-2}\langle \lambda_{i}e_{i},X\rangle\langle X, E_{l}\rangle+|X|^{\frac{\alpha}{2}}\langle \lambda_{i}e_{i},E_{l}\rangle\right)^{2}\\
&=\sum_{i}\frac{\partial \sigma_{n}}{\partial \lambda_{i}}\lambda^{2}_{i}\left[|X|^{\alpha}+\left(\frac{\alpha^{2}}{4}+\alpha\right)|X|^{\alpha-2}\langle e_{i},X\rangle^{2}\right]\\
&=|X|^{\alpha}\sigma_{1}\sigma_{n}+\left(\frac{\alpha^{2}}{4}+\alpha\right)|X|^{\alpha-1}\langle \nabla h, \nabla |X|\rangle \sigma_{n}.
\end{split}
\end{equation}
Then, combining \eqref{cwq} with \eqref{dx}, and applying $dV_{n}=h\sigma_{n}d\sigma$, we derive
\begin{equation*}
\begin{split}
\label{Up3}
&n\left[ \int_{\sn}|X|^{\alpha+2}dV_{n}-\frac{\Big|\int_{\sn}|X|^{\frac{\alpha}{2}}X dV_{n}\Big|^{2}}{\int_{\sn}dV_{n}} \right]\\
&\leq\int_{\sn}h^{2}|X|^{\alpha}\sigma_{1}\sigma_{n}d\sigma+\left(\frac{\alpha^{2}}{4}+\alpha\right)\int_{\sn}|X|^{\alpha-1}h^{2}\sigma_{n}\langle \nabla h, \nabla |X|\rangle d\sigma\\
&=\int_{\sn}|X|^{\alpha}h(\Delta h+nh)dV_{n} +\left(\frac{\alpha^{2}}{4}+\alpha\right)\int_{\sn}|X|^{\alpha-1}h \langle \nabla h, \nabla |X|\rangle d V_{n}.
\end{split}
\end{equation*}
The proof is complete.
\end{proof}

Using Lemma \ref{IN}, we get the following result, which is the main ingredient of proving Theorem \ref{main}.
\begin{prop}\label{ab}
Let $K$ be origin-symmetric. Let $m, M>0$. Suppose $n-3\leq q \leq n+1$. Let $m\leq h |Dh|^{q-(n+1)}\frac{1}{\kappa}\leq M$. Then we have \[
n\int_{\sn}|Dh|^{2}d\sigma\leq \frac{M}{m}\int_{\sn}h(\Delta h+nh)d\sigma.
\]
\end{prop}
\begin{proof}
Let $\alpha: =q-(n+1)$ in Lemma \ref{IN}. It follows that $\frac{\alpha^{2}}{4}+\alpha\leq 0$ for $n-3\leq q \leq n+1$. Since $|X|\langle \nabla h, \nabla |X|\rangle=\sum_{i}\lambda_{i}h^{2}_{i}\geq c_{0} |\nabla h|^{2}$ where $c_{0}>0$ depends on $\partial K$. Set $dC_{q}=h |Dh|^{q-(n+1)}\frac{1}{\kappa}d\sigma$, thus we have
\[
n\int_{\sn}|Dh|^{2}dC_{q}\leq \int_{\sn}h(\Delta h+nh)dC_{q}.
\]
Based on the assumption, we get
\[
mn\int_{\sn}|Dh|^{2}d\sigma\leq  M\int_{\sn}h(\Delta h+nh)d\sigma.
\]
The proof is complete.
\end{proof}
  Drawing inspiration from \cite{HI25}, we are in a position to prove Theorem \ref{main} via Proposition \ref{ab}.

\begin{proof}[Proof of Theorem \ref{main}.]

Set
\[
M=\max_{\sn} \left(|Dh|^{q-(n+1)}\frac{h}{\kappa}\right), \ \ \   m=\min_{\sn} \left(|Dh|^{q-(n+1)}\frac{h}{\kappa}\right), \ \ \ \varepsilon:=\frac{M}{m}-1.
\]
In view of Proposition \ref{ab}, integration of parts yields
\begin{equation}\label{cv}
(n+1+\varepsilon)\int_{\sn}|\nabla h|^{2}d\sigma\leq n\varepsilon \int_{\sn}h^{2}d\sigma.
\end{equation}
Applying the Poincar\'{e} inequality on $\sn$ to $h$, we obtain
\begin{equation}\label{cv2}
n\int_{\sn}\left(h-\frac{\int_{\sn} h d\sigma}{\int_{\sn}d\sigma} \right)^{2}d\sigma  \leq \int_{\sn}|\nabla h|^{2}d\sigma.
\end{equation}
Combining \eqref{cv} and \eqref{cv2}, it yields
\[
\frac{\int_{\sn}\left( h-\frac{\int hd \sigma }{\int_{\sn}d\sigma} \right)^{2}d\sigma}{\int_{\sn}d\sigma}\leq \frac{\varepsilon}{n+1+\varepsilon}\frac{\int_{\sn} h^{2}d\sigma}{\int_{\sn}d\sigma} \leq \frac{\varepsilon}{n+1}\frac{\int_{\sn} h^{2}d\sigma}{\int_{\sn}d\sigma}.
\]
It follows that
\begin{equation}\label{dd}
\frac{\int_{\sn}\left(\frac{h}{\int_{\sn} hd\sigma/\int_{\sn}d\sigma} -1\right)^{2}d\sigma}{\int_{\sn}d\sigma}\leq \frac{\varepsilon}{n+1}\frac{\frac{\int_{\sn} h^{2}d\sigma}{\int_{\sn}d\sigma}}{(\int_{\sn} hd\sigma/\int_{\sn}d\sigma)^{2}}.
\end{equation}
We now show that the right-hand side of \eqref{dd} is bounded. Let $R_{h}=\max_{\sn}h$. Assume that $R_{h}=h(w)$ for a unit vector $w\in \sn$. By the convexity of hypersurface, for any $x\in \sn$, we have
\[
h(x)\geq \langle x, w\rangle R_{h}.
\]
Building on this fact, then there exists a positive constant $c_{1}$, depending on $n$, to satisfy
\begin{equation}\label{io}
\frac{\int_{\sn} hd\sigma}{\int_{\sn}d\sigma}\geq \frac{1}{\int_{\sn}d\sigma}\int_{\langle x,w\rangle\geq 1/2}hd\sigma\geq \frac{R_{h}}{2\int_{\sn}d\sigma}\int_{\langle x,w\rangle\geq 1/2}d\sigma\geq c_{1}R_{h},
\end{equation}
then it suffices to get
\begin{equation}\label{bbc}
\left( \frac{\int_{\sn}h^{2}d\sigma}{\int_{\sn}d\sigma} \right)^{\frac{1}{2}}\leq R_{h}\leq \frac{1}{c_{1}}\frac{\int_{\sn} hd\sigma}{\int_{\sn}d\sigma}.
\end{equation}
Substituting \eqref{bbc} into \eqref{dd}, we obtain
\[
\frac{\int_{\sn}\left(\frac{h}{\int_{\sn} hd\sigma/\int_{\sn}d\sigma} -1\right)^{2}d\sigma}{\int_{\sn}d\sigma}\leq \frac{\varepsilon}{(n+1)c^{2}_{1}}.
\]
Therefore the proof is finished.
\end{proof}

\section{Existence and uniqueness of solutions to the dual Minkowski problem}
\label{Sec4}

We first need to obtain the a priori estimates of solutions to \eqref{Mong} with the aid of Theorem \ref{main}.

\begin{lem}\cite[Lemma 7.6.4]{S14}\label{main2}
Let $K_{1}, K_{2}$ be two convex bodies in $\rnnn$. Then the following fact holds:
\begin{equation}\label{Sch}
\delta_{2}(K_{1},K_{2})^{2}\geq \alpha_{n} {\rm diam}(K_{1}\cup K_{2})^{-n}\delta_{H}(K_1,K_2)^{n+2},
\end{equation}
where $\alpha_{n}$ is a dimensional constant and ${\rm diam}(K_{1}\cup K_{2})$ is the diameter of the set $K_{1}\cup K_{2}$.
\end{lem}
Based on Lemma \ref{main2} and Theorem \ref{main}, we derive the $C^{0}$ estimate as follows.
\begin{lem}\label{co}

Suppose that $q$ satisfies either $ 0<q\leq n+1$ if $ \ 1\leq n\leq 3$,  or $ n-3\leq q\leq n+1$ if $ n>3$. Assume that $K$ is a smooth, origin-symmetric, strictly convex body such that
\begin{equation}\label{Du}
1-\varepsilon\leq |Dh|^{q-(n+1)}\frac{h}{\kappa}\leq 1+\varepsilon
\end{equation}
for some $\varepsilon\in (0,\varepsilon_{0})$ with $\varepsilon_{0}>0$, then there exists $C=C(\varepsilon_{0},n)>0$ such that
\begin{equation}\label{hco}
1/C\leq h(x)\leq C, \quad \forall x\in \sn,
\end{equation}
and
\begin{equation}\label{pco}
1/C\leq \rho(u)\leq C, \quad \forall u\in \sn,
\end{equation}
where $h(x)$ and $\rho(u)$ are the support function and the radial function of $K$ respectively.
\end{lem}
\begin{proof}

Let $u$ and $x$ be related by $\rho(u)u=\nabla h(x)+h(x)x=Dh(x)$. Clearly, we have
\[
\min_{\sn}h(x)\leq \rho(u)\leq \max_{\sn}h(x).
\]
This implies that the validity of \eqref{hco} is equivalent  to that of \eqref{pco}. So we only need to establish \eqref{hco}.

From \eqref{io}, for $c_{1}>0$, we know that
\begin{equation*}\label{}
h_{\bar{K}}(x)=\frac{h(x)}{\int_{\sn} hd\sigma/\int_{\sn}d\sigma}\leq \frac{1}{c_{1}}, \quad \forall x\in \sn.
\end{equation*}
It follows that
\begin{equation}\label{az}
{\rm diam} (\bar{K}\cup B_{1})\leq 2\left( 1+\frac{1}{c_{1}} \right).
\end{equation}
By applying Theorem \ref{main}, we obtain
\begin{equation}\label{az2}
\delta_{2}(\bar{K},B_{1})\leq \beta \varepsilon^{\frac{1}{2}}_{0}.
\end{equation}
Combining \eqref{az} and \eqref{az2} with \eqref{Sch}, we conclude that there exists a constant $c_{2}>0$, depending only on $n$, such that
\begin{equation}\label{az3}
\delta_{H}(\bar{K},B_{1})\leq \alpha^{-\frac{1}{n+2}}_{n}{\rm diam}(\bar{K}\cup B_{1})^{\frac{n}{n+2}}\delta_{2}(\bar{K}, B_{1})^{\frac{2}{n+2}}\leq c_{2}\varepsilon^{\frac{1}{n+2}}_{0}.
\end{equation}
Eq. \eqref{az3} gives
\[
1-c_{2}\varepsilon^{\frac{1}{n+2}}_{0}\leq \frac{h}{\int_{\sn}hd\sigma/\int_{\sn}d\sigma}\leq 1+c_{2}\varepsilon^{\frac{1}{n+2}}_{0}.
\]
Furthermore, for $\varepsilon_{0}$ equipped with $c_{2}\varepsilon^{\frac{1}{n+2}}_{0}<1$, there is
\begin{equation}\label{radio}
\frac{\max_{\sn} h}{\min_{\sn} h}\leq \frac{1+c_{2}\varepsilon^{\frac{1}{n+2}}_{0}}{1-c_{2}\varepsilon^{\frac{1}{n+2}}_{0}}.
\end{equation}
Now, employing \eqref{Du}, for $q>0$, we obtain
\begin{equation}\label{}
(\min_{\sn} h)^{q}\int_{\sn}d\sigma \leq \int_{\sn}|Dh|^{q-(n+1)}\frac{h}{\kappa}d\sigma\leq (1+\varepsilon_{0})\int_{\sn}d\sigma.
\end{equation}
This illustrates that there exists a constant $C_{1}>0$ such that
\begin{equation}\label{infi}
\min_{\sn} h\leq C_{1}.
\end{equation}
Substituting \eqref{infi} into \eqref{radio}, then there is a positive constant $C_{2}>0$ depending on $n, \varepsilon_{0}$ so that
\[
\max_{\sn} h\leq C_{2}.
\]
 Similarly, for $q>0$,
\begin{equation}\label{}
(\max_{\sn} h)^{q}\int_{\sn}d\sigma \geq \int_{\sn}|Dh|^{q-(n+1)}\frac{h}{\kappa}d\sigma\geq (1-\varepsilon_{0})\int_{\sn}d\sigma.
\end{equation}
It yields that for a positive constant $C_{3}>0$, there is
\begin{equation}\label{infi2}
\max_{\sn} h\geq C_{3}.
\end{equation}
Applying \eqref{infi2} into \eqref{radio},  we also have
\[
\min_{\sn}h\geq C_{4}
\]
for a positive constant $C_{4}>0$ depending on $n, \varepsilon_{0}$. This completes the proof.
\end{proof}
The $C^{1}$ estimate follows from the $C^{0}$ estimate above and the convexity of the hypersurface.
\begin{lem}\label{C1}
Suppose that $q$ satisfies either $0<q\leq n+1$ if $ \ 1\leq n\leq 3$,  or $ n-3\leq q\leq n+1$ if $ n>3$.  Let $\alpha\in (0,1)$.  Let $f$ be an even, smooth and positive function on $\sn$, and $K$ be a smooth, origin-symmetric and strictly convex solution to Eq. \eqref{Mong}. There exists a constant $\varepsilon_{0}>0$ depending only on $n$, $\alpha$ such that if $||f-1||_{C^{\alpha}}\leq \varepsilon$ for some $\varepsilon\in (0,\varepsilon_{0})$, then there is a constant $C>0$ depending on $\varepsilon_{0},n$, such that
\begin{equation*}\label{C12}
|\nabla h(x)|\leq  C, \quad \forall x\in \sn,
\end{equation*}
and
\begin{equation*}\label{C13}
|\nabla \rho(u)|\leq  C, \quad \forall u\in \sn.
\end{equation*}
\end{lem}
\begin{proof}
Due to $\rho(u)u=\nabla h(x)+h(x)x=Dh(x)$, we have
\[
h=\frac{\rho^{2}}{\sqrt{|\nabla \rho|^{2}+\rho^{2}}}, \quad \rho^{2}=h^{2}+|\nabla h|^{2}.
\]
Hence, combining the above facts with Lemma \ref{co}, we obtain the desired result.
\end{proof}

Next, our goal is to obtain the $C^{2}$ estimate of  solutions to \eqref{Mong}.

\begin{lem}\label{C2}
Suppose that $q$ satisfies either $ 0<q\leq n+1$ if $ \ 1\leq n\leq 3$,  or $ n-3\leq q\leq n+1$ if $ n>3$.  Let $\alpha\in (0,1)$.  Let $f$ be an even, smooth and positive function on $\sn$, and $K$ be a smooth, origin-symmetric and strictly convex solution to \eqref{Mong}. There exists a constant $\varepsilon_{0}>0$ depending only on $n$, $\alpha$ such that if $||f-1||_{C^{\alpha}}\leq \varepsilon$ for some $\varepsilon\in (0,\varepsilon_{0})$, then for some positive constant $C$ depending on $\varepsilon_{0}, n$, the principal curvatures $\kappa_{1},\ldots, \kappa_{n}$ of $\partial K$ satisfy
\begin{equation*}
1/C\leq \kappa_{i}(x)\leq C, \quad \forall x\in \sn, \ i=1,\ldots,n.
\end{equation*}

\end{lem}
\begin{proof}
The proof is divided into two parts, in the first part, we derive an upper bound for the Gauss curvature $\kappa(x)$. In the second part, we derive an upper bound for the principal radii of curvature $b_{ij}=h_{ij}+h\delta_{ij}$.

Step I: Based on the assumption, for some $\varepsilon\in (0,\varepsilon_{0})$ with $\varepsilon_{0}>0$, we have
\begin{equation}\label{Du2}
1-\varepsilon\leq \rho^{q-(n+1)}\frac{h}{\kappa}\leq 1+\varepsilon.
\end{equation}
By  Lemma \ref{co}, and using \eqref{Du2}, we have
\begin{equation}\label{}
\kappa\leq \rho^{q-(n+1)}h\frac{1}{1-\varepsilon_{0}}\leq C_{0}
\end{equation}
for a positive constant $C_{0}$.

Step II: Set the auxiliary function as
\begin{equation*}
Q(x)={\rm log}{\rm tr}(\{b_{ij}\})-A{\rm log}h+B|\nabla h|^{2},
\end{equation*}
where ${\rm tr}(\{b_{ij}\})$ is the sum of the eigenvalues of matrix $\{b_{ij}\}$, $A$ and $B$ are positive constants to be chosen later. Assume $\max_{\sn}Q(x)$ is attained at a point $x_{0}\in \sn$. By a rotation, we may assume $\{b_{ij}\}(x_{0})$ is diagonal.
  Then we have at $x_{0}$,
\begin{equation}
\begin{split}
\label{wq}
0=\nabla_{i}Q&=\frac{1}{\sum_{j}b_{jj}}\sum_{j}\nabla_{i}b_{jj}-A\frac{h_{i}}{h}+2B\sum_{k}h_{k}h_{ki}\\
&= \frac{1}{\sum_{j}b_{jj}}\sum_{j}(h_{jji}+h_{i})-A\frac{h_{i}}{h}+2Bh_{i}h_{ii}\\
&= \frac{1}{\sum_{j}b_{jj}}\sum_{j}(h_{ijj}+h_{j}\delta_{ij}-h_{i})+\frac{1}{\sum_{j}b_{jj}}\sum_{j}h_{i}-A\frac{h_{i}}{h}+2Bh_{i}h_{ii}\\
&=\frac{1}{\sum_{j}b_{jj}}\sum_{j}(h_{ijj}+h_{j}\delta_{ij})-A\frac{h_{i}}{h}+2Bh_{i}h_{ii},
\end{split}
\end{equation}
where we used the fact that $h_{ijk}-h_{ikj}=h_{j}\delta_{ik}-h_{k}\delta_{ij}$, and there holds
\begin{equation*}
\begin{split}
&0\geq \nabla_{ii}Q=\frac{1}{\sum_{j}b_{jj}}\sum_{j}\nabla_{ii}b_{jj}-\frac{1}{(\sum_{j}b_{jj})^{2}}(\sum_{j}\nabla_{i}b_{jj})^{2}-A\left( \frac{h_{ii}}{h} -\frac{h^{2}_{i}}{h^{2}}\right)+2B\left( \sum_{k}h_{k}h_{kii}+h^{2}_{ii} \right).
\end{split}
\end{equation*}
At $x_{0}$, we also have
\begin{equation}
\begin{split}
\label{II}
0&\geq b^{ij}Q_{ij}\\
&=\sum_{i}b^{ii}\frac{1}{\sum_{j}b_{jj}}\sum_{j}\nabla_{ii}b_{jj}-\sum_{i}\frac{1}{(\sum_{j}b_{jj})^{2}}b^{ii}(\sum_{j}\nabla_{i}b_{jj})^{2}-A\sum_{i}b^{ii}\left( \frac{h_{ii}}{h} -\frac{h^{2}_{i}}{h^{2}}\right)\\
&\quad+2B\sum_{i}b^{ii}\sum_{k}h_{k}h_{kii}+2B\sum_{i}b^{ii}h^{2}_{ii}\\
&\geq \sum_{i}b^{ii}\frac{1}{\sum_{j}b_{jj}}\sum_{j}\nabla_{ii}b_{jj}-\sum_{i}\frac{1}{(\sum_{j}b_{jj})^{2}}b^{ii}(\sum_{j}\nabla_{i}b_{jj})^{2}-A\sum_{i}\frac{h_{ii}+h}{h}b^{ii}+A\sum_{i}b^{ii}\\
&\quad+2B\sum_{i}b^{ii}\sum_{k}h_{k}h_{kii}+2B\sum_{i}b^{ii}(b_{ii}-h)^{2}.
\end{split}
\end{equation}
The Ricci identity on sphere reads
\begin{equation*}
\nabla_{kk}b_{ij}=\nabla_{ij}b_{kk}-\delta_{ij}b_{kk}+\delta_{kk}b_{ij}-\delta_{ik}b_{jk}+\delta_{jk}b_{ik}.
\end{equation*}
Then \eqref{II} becomes
\begin{equation}
\begin{split}
\label{I2}
0&\geq \sum_{i}\frac{1}{\sum_{j}b_{jj}}b^{ii}\sum_{j}(\nabla_{jj} b_{ii}+b_{jj}-b_{ii})-\sum_{i}\frac{1}{(\sum_{j}b_{jj})^{2}}b^{ii}(\sum_{j}\nabla_{i}b_{jj})^{2}-\frac{nA}{h}+A\sum_{i}b^{ii}\\
&\quad +2B\sum_{k}h_{k}\sum_{i}b^{ii}h_{kii}+2B\sum_{i}b_{ii}-4nBh.
\end{split}
\end{equation}
Since
\begin{equation}\label{rt}
\log h=-\log {\rm det}(\nabla^{2}h+hI)+\log (f\rho^{n+1-q}).
\end{equation}
Set $\Phi:=\log (f\rho^{n+1-q})$. Differentiating \eqref{rt}, at $x_{0}$, it gives
\begin{equation}
\begin{split}
\label{tr}
\frac{h_{j}}{h}&=-\sum_{i,k}b^{ik}\nabla_{j}b_{ik}+\nabla_{j}\Phi\\
&=-\sum_{i,k}b^{ik}(h_{ikj}+h_{j}\delta_{ik})+\nabla_{j}\Phi\\
&=-\sum_{i}b^{ii}h_{iij}-h_{j}\sum_{i}b^{ii}+\nabla_{j}\Phi\\
&=-\sum_{i}b^{ii}(h_{iji}+h_{i}\delta_{ij}-h_{j})-h_{j}\sum_{i}b^{ii}+\nabla_{j}\Phi\\
&=-\sum_{i}b^{ii}(h_{jii}+h_{i}\delta_{ij})+\nabla_{j}\Phi,
\end{split}
\end{equation}
where $b^{ij}$ is the inverse of $b_{ij}$, and
\begin{equation}\label{re}
\frac{h_{jj}}{h}-\frac{h^{2}_{j}}{h^{2}}=-\sum_{i}b^{ii}\nabla_{jj}b_{ii}+\sum_{i,k}b^{ii}b^{kk}(\nabla_{j}b_{ik})^{2}+\nabla_{jj}\Phi.
\end{equation}
Besides, for each $i$, there is
\begin{equation}
\begin{split}
\label{AMI}
&\sum_{j}b_{jj}\sum_{j,k}b^{kk}(\nabla_{j}b_{ik})^{2}\\
&\geq \sum_{j}b_{jj}\sum_{j}b^{jj}(\nabla_{i}b_{jj})^{2}\\
&\geq \left(\sum_{j}\sqrt{b_{jj}b^{jj}(\nabla_{i}b_{jj})^{2}}\right)^{2}\\
&=(\sum_{j}|\nabla_{i}b_{jj}|)^{2}\\
&\geq (\sum_{j}\nabla_{i}b_{jj})^{2}.
\end{split}
\end{equation}
Employing \eqref{AMI}, one sees
\begin{equation}\label{AM2}
\sum_{i}b^{ii}\frac{1}{(\sum_{j}b_{jj})^{2}}(\sum_{j}\nabla_{i}b_{jj})^{2}-\frac{1}{\sum_{j}b_{jj}}\sum_{i,j,k}b^{ii}b^{kk}(\nabla_{j}b_{ik})^{2}\leq 0.
\end{equation}
Now, substituting \eqref{re} into \eqref{I2}, employing \eqref{tr} and \eqref{AM2}, we get
\begin{equation}
\begin{split}
\label{po}
0&\geq \frac{1}{\sum_{j}b_{jj}}\sum_{j}\left(-\frac{h_{jj}}{h}+\frac{h^{2}_{j}}{h^{2}} \right)+\frac{1}{\sum_{j}b_{jj}}\sum_{i,j,k}b^{ii}b^{kk}(\nabla_{j}b_{ik})^{2}+\frac{1}{\sum_{j}b_{jj}}\sum_{j}\nabla_{jj}\Phi -\frac{n^{2}}{\sum_{j}b_{jj}} \\ &\quad-\sum_{i}\frac{1}{(\sum_{j}b_{jj})^{2}}b^{ii}(\sum_{j}\nabla_{i}b_{jj})^{2}
-\frac{nA}{h}+A\sum_{i}b^{ii}+2B\sum_{k}h_{k}\sum_{i}b^{ii}h_{kii}+2B\sum_{i}b_{ii}-4nBh\\
&\geq -\frac{1}{h}-\frac{nA}{h}+\frac{1}{\sum_{j}b_{jj}}\sum_{j}\nabla_{jj}\Phi+A\sum_{i}b^{ii}+2B\sum_{k}h_{k}\left( -\frac{h_{k}}{h} -b^{kk}h_{k}+\nabla_{k}\Phi\right)\\
& \quad+2B\sum_{i}b_{ii}-\frac{n^{2}}{\sum_{j}b_{jj}}-4nBh\\
& \geq -\frac{1}{h}-\frac{nA}{h}-\frac{2B|\nabla h|^{2}}{h}+(A-2B|\nabla h|^{2})\sum_{i}b^{ii}+\frac{1}{\sum_{j}b_{jj}}\sum_{j}\nabla_{jj}\Phi\\
&\quad +2B\sum_{k}h_{k}\nabla_{k}\Phi+2B\sum_{i}b_{ii}-\frac{n^{2}}{\sum_{j}b_{jj}}-4nBh.
\end{split}
\end{equation}
From $\Phi=\log(f\rho^{n+1-q})$, we get
\begin{equation}
\begin{split}
\label{q1}
2B\sum_{k}h_{k}\nabla_{k}\Phi=2B\sum_{k}h_{k}\left( \frac{f_{k}}{f}+(n+1-q)\frac{hh_{k}+h_{k}h_{kk}}{\rho^{2}}\right),
\end{split}
\end{equation}
and
\begin{equation}
\begin{split}
\label{q2}
\frac{1}{\sum_{j}b_{jj}}\sum_{j}\nabla_{jj}\Phi&=\frac{1}{\sum_{j}b_{jj}}\sum_{j}\left(\frac{ff_{jj}-f^{2}_{j}}{f^{2}} +(n+1-q)\frac{hh_{jj}+h^{2}_{j}+h^{2}_{jj}+\sum_{k}h_{k}h_{kjj}}{\rho^{2}} \right)\\
&\quad -2\frac{1}{\sum_{j}b_{jj}}(n+1-q)\sum_{j}\frac{(hh_{j}+h_{j}h_{jj})^{2}}{\rho^{4}}.
\end{split}
\end{equation}
Using \eqref{wq}, \eqref{q1}, \eqref{q2}, Lemmas \ref{co} and \ref{C1}, then we get
\begin{equation}
\begin{split}
\label{ws}
&2B\sum_{k}h_{k}\nabla_{k}\Phi+\frac{1}{\sum_{j}b_{jj}}\sum_{j}\nabla_{jj}\Phi\\
&\geq -C_{0}- C_{1}B+(n+1-q)\sum_{k}\frac{h_{k}}{\rho^{2}}\left[2Bh_{k}h_{kk}+\frac{1}{\sum_{j}b_{jj}}\sum_{j}h_{kjj}\right]-C_{2}\frac{1}{\sum_{j}b_{jj}}\\
&\quad -(n+1-q)\frac{1}{\sum_{j}b_{jj}}\sum_{j} \left[ \frac{h|h_{jj}|+h^{2}_{jj}}{\rho^{2}}+2\frac{h^{2}_{j}h^{2}_{jj}}{\rho^{4}} \right]\\
&=-C_{0}-C_{1}B+(n+1-q)\sum_{k}\frac{h_{k}}{\rho^{2}}\left(-\frac{1}{\sum_{j}b_{jj}}\sum_{j}h_{j}\delta_{kj}+A\frac{h_{k}}{h}\right)-C_{2}\frac{1}{\sum_{j}b_{jj}} \\
&\quad -(n+1-q)\frac{1}{\sum_{j}b_{jj}}\sum_{j}\frac{h|b_{jj}-h|+(b_{jj}-h)^{2}}{\rho^{2}}-2(n+1-q)\frac{1}{\sum_{j}b_{jj}}\sum_{j}\frac{h^{2}_{j}(b_{jj}-h)^{2}}{\rho^{4}}\\
&\geq- \tilde{C}_{0}-\tilde{C}_{1}A-\tilde{C}_{2}B-\tilde{C}_{3}\frac{1}{\sum_{j}b_{jj}}-(n+1-q)\frac{\rho^{2}+2|\nabla h|^{2}}{\rho^{4}}\sum_{j}b_{jj}
\end{split}
\end{equation}
for the positive constants $C_{0}, C_{1}, C_{2},\tilde{C}_{0}, \tilde{C}_{1}, \tilde{C}_{2}, \tilde{C}_{3}$ depending only on the constants from Lemmas \ref{co} and \ref{C1}. Now we take $A=2B \max_{\sn} |\nabla h|^{2}+1$, and
\[
B=(n+1-q)\frac{\max_{\sn} \rho^{2}+2\max_{\sn} |\nabla h|^{2}}{\min_{\sn} \rho^{4}}+1.
\]
Then applying \eqref{ws} into \eqref{po}, we obtain
\begin{equation*}
\begin{split}
0\geq -\frac{(n+1)A}{h}-\tilde{C}_{0}-\tilde{C}_{1}A-\tilde{C}_{2}B-(\tilde{C}_{3}+n^{2})\frac{1}{\sum_{j}b_{jj}}+B\sum_{i}b_{ii}-4nhB.
\end{split}
\end{equation*}
Thus when $\sum_{i}b_{ii}\gg1$, we get
\[
\sum_{i}b_{ii}\leq C
\]
for a positive constant $C$. The proof is complete.
\end{proof}
By the Evans-Krylov theorem and the Schauder regularity theory, together with the a priori estimates in Lemmas \ref{co}, \ref{C1} and \ref{C2}, we obtain the following theorem.
\begin{theo}\label{UI}
Suppose that $q$ satisfies either $ 0<q\leq n+1$ if $ \ 1\leq n\leq 3$,  or $ n-3\leq q\leq n+1$ if $ n>3$.  Let $\alpha\in (0,1)$. Let $f$ be an even, smooth and positive function on $\sn$, and $K$ be a smooth, origin-symmetric and strictly convex solution to \eqref{Mong}. There exists a constant $\varepsilon_{0}>0$ depending only on $n$, $\alpha$ such that if $||f-1||_{C^{\alpha}}\leq \varepsilon$ for some $\varepsilon\in (0,\varepsilon_{0})$,  then for any $\ell\geq 2$ and $\gamma\in (0,1)$, there is a constant $C>0$ depending on $\varepsilon_{0}, n$ such that
\begin{equation*}
||h||_{C^{\ell+1,\gamma}}\leq C.
\end{equation*}
\end{theo}
We are in a position to prove Theorem \ref{coro2}.

\begin{proof}[Proof of Theorem \ref{coro2}.]

We first demonstrate the existence part by employing degree-theoretic arguments. Let $T_{t}(\cdot): C^{4,\alpha}(\sn)\rightarrow C^{2,\alpha}(\sn)$ be a nonlinear differential operator, which is defined as
\begin{equation*}\label{}
T_{t}(h)=\det(\nabla^{2}h+hI)-h^{-1}(|\nabla h|^{2}+h^{2})^{\frac{n+1-q}{2}}f_{t}
\end{equation*}
for $t\in [0,1]$, where
\[
f_{t}=(1-t)+tf.
\]
For $R>0$ fixed, define $\mathcal{O}\subset C^{4,\alpha}(\sn)$ by
\[
\mathcal{O}=\{h\in C^{4,\alpha}(\sn): \ h(x)=h(-x), \ \forall x\in \sn, \ {\rm and} \ ||h||_{C^{4,\alpha}(\sn)}< R  \}.
\]
By Theorem \ref{UI}, one sees that $T_{t}(h)=0$ has no solution on $\partial \mathcal{O}$ if $R$ is sufficiently large. Therefore, the degree of $T_{t}$ is well defined (see, e.g. \cite[Section 2]{LY89} or \cite[Section 3]{GM06} ). Since degree is homotopic invariant,
\[
\deg(T_{0}, \mathcal{O}, 0)=\deg(T_{1}, \mathcal{O}, 0).
\]
On the one hand, at $t=0$, by Theorem \ref{ori}, $h=1$ is the unique solution of $\eqref{Mong}$ when $f=1$.  On the other hand, since $T$ is symmetric, it is clear to show that the linearized operator of $T_{0}$  at $h=1$ is
\[
L_{0}\eta=\Delta \eta+q\eta
\]
for even $\eta\in C^{2}(\sn)$, when $0<q< n$,
we know that $q$ is not an eigenvalue of $(-\Delta)$ on $\sn$; furthermore, for the case $q=n$,  if $L_{0}\eta=0$, then $\eta$ are linear functions of $\rnnn$,  i.e., $\eta\in {\rm Span}\{x_{1},\ldots, x_{n+1}\}$ that are odd, due to the evenness of $\eta$, one sees that $\eta=0$. Thus the linearized operator $L_{0}$ is invertible for $0<q\leq n$. Based on this fact, we compute the degree by means of formula
\[
{\rm deg}(T_{0},\mathcal{O},0)=\sum_{\mu_{j}>0}(-1)^{\zeta_{j}},
\]
where $\mu_{j}>0$ are the eigenvalues of the linearized operator of $T_{0}$ and $\zeta_{j}$ denotes its multiplicities. Since the eigenvalues of the Beltrami-Laplace operator $\Delta$ on $\sn$ are strictly less than $-n$, except for the first two eigenvalues $0$ and $-n$, it follows that for $0<q\leq n$, there is only one positive eigenvalue of $L_{0}$ with multiplicity 1, namely $\mu=q$. Consequently,
\[
\deg(T_{0}, \mathcal{O}, 0)=\deg(T_{1}, \mathcal{O}, 0)=-1\neq 0.
\]
Hence, there is an even solution to \eqref{Mong}. The regularity of $h$ follows directly from Theorem \ref{UI}.

The following lemma is essential for proving the uniqueness part.
\begin{lem}\label{PRE}
Suppose that $q$ satisfies either $ 0<q\leq n$ if $ \ 1\leq n\leq 3$,  or $ n-3\leq q\leq n$ if $ n>3$. Let $0<\alpha<1$. Let $f$ be an even, smooth and positive function on $\sn$. There exists a small constant $\varepsilon_{0}>0$ such that if $||f-1||_{C^{\alpha}}\leq\varepsilon_{0}$, $||h_{K}-1||_{\infty}\leq \varepsilon_{0}$ and $||h_{L}-1||_{\infty}\leq \varepsilon_{0}$, where $K$ and $L$ are smooth, origin-symmetric and strictly convex bodies satisfying Eq. \eqref{Mong}, then $K=L.$
\end{lem}
\begin{proof}
By Theorem \ref{UI}, we find
\[
||h_{K}||_{C^{2,\alpha}}\leq C_{0}, \quad {\rm and} \quad ||h_{L}||_{C^{2,\alpha}}\leq C_{0},
\]
where $C_{0}>0$ is a constant depending only on $\varepsilon_{0}$. Moreover,
\begin{equation*}\label{}
\Big|\Big| \frac{f\rho^{n+1-q}_{K}}{h_{K}}-1\Big|\Big|_{C^{\alpha}}\leq \hat{\varepsilon}_{0}\rightarrow 0 \quad as \quad \varepsilon_{0}\rightarrow 0,
\end{equation*}
and
\begin{equation*}\label{}
\Big|\Big| \frac{f\rho^{n+1-q}_{L}}{h_{L}}-1\Big|\Big|_{C^{\alpha}}\leq \hat{\varepsilon}_{0}\rightarrow 0 \quad as \quad \varepsilon_{0}\rightarrow 0.
\end{equation*}
It is clear to show
\begin{equation}
\begin{split}
\label{Mong2}
\frac{f\rho^{n+1-q}_{K}}{h_{K}}-1&=\det(\nabla^{2}h_{K}+h_{K}I)-\det(\nabla^{2}1+I)\\
&=\int^{1}_{0}\frac{d}{dt}\det(\nabla^{2}((1-t)+th_{K})+((1-t)+th_{K})I)dt\\
&=\sum^{n}_{i,j=1}\int^{1}_{0}U^{ij}_{t}dt\cdot ((h_{K}-1)_{ij}+(h_{K}-1)\delta_{ij})\\
&=\sum^{n}_{i,j=1}a_{ij}((h_{K}-1)_{ij}+(h_{K}-1)\delta_{ij}),
\end{split}
\end{equation}
where the coefficient $a_{ij}=\int^{1}_{0}U^{ij}_{t}dt$, and $U^{ij}_{t}$ is the cofactor matrix of
\[
\nabla^{2}((1-t)+th_{K})+((1-t)+th_{K})I.
\]
Since
\[
||h_{K}||_{C^{2,\alpha}}\leq C_{0},
\]
and there exists a positive constant $C_{1}>0$ such that,
\[
\frac{1}{C_{1}}I\leq \{a_{ij}\}\leq C_{1}I.
\]
This illustrates that $\eqref{Mong2}$ is uniformly elliptic. Applying the Schauder estimate (see, e.g. \cite[Chapter 6]{GT01}) to $(h_{K}-1)$, one sees that for a universal positive constant $C$,
\[
||h_{K}-1||_{C^{2,\alpha}}\leq C\left(||h_{K}-1||_{\infty}+\Big|\Big| \frac{f\rho^{n+1-q}_{K}}{h_{K}}-1\Big| \Big|_{C^{\alpha}}\right)\leq C(\varepsilon_{0}+\hat{\varepsilon}_{0}).
\]
Along the same argument, we also derive
\[
||h_{L}-1||_{C^{2,\alpha}}\leq C(\varepsilon_{0}+\hat{\varepsilon}_{0}).
\]
So $K, L$ lie in the $C^{2}$-neighbourhood of $B_{1}$. On the other hand, as mentioned above, the linearized operator $L_{0}$ of \eqref{Mong} at $h=1$ is invertible for $0<q\leq n$. Since $K$ and $L$ satisfy \eqref{Mong}, by means of the inverse function theorem, provided that $\varepsilon_{0}$ is sufficiently small, we have $K=L$.
\end{proof}
Now, utilizing Lemma \ref{PRE}, we verify the uniqueness part of Theorem \ref{coro2}. Assume, to the contrary, that there are two different solutions $K_{i}$ and $L_{i}$ for $i\in \N$, by Lemma \ref{PRE}, we conclude that at least one of them, say $K_{i}$, then there exist $f_{i}$ and $K_{i}$ such that
\[
h_{K_{i}}\det(\nabla^{2}h_{K_{i}}+h_{K_{i}}I)\rho^{q-(n+1)}_{K_{i}}=f_{i},
\]
and satisfy
\[
||h_{K_{i}}-1||_{\infty}>\varepsilon_{0}, \quad ||f_{i}-1||_{C^{\alpha}}\rightarrow 0 \quad as \ \ i\rightarrow \infty.
\]
Meanwhile, with the aid of  Theorem \ref{UI}, for $\ell\geq 2$ and $\gamma\in (0,1)$, one gets
\[
||h_{K_{i}}||_{C^{\ell+1,\gamma}}\leq C
\]
for some positive constant $C$ depending only on $\varepsilon_{0}, n$. Then by the Arzel\`a-Ascoli theorem, passing to a subsequence,  there exists a smooth, origin-symmetric and strictly convex body $\widetilde{K}$ such that $h_{K_{i}}\rightarrow h_{\widetilde{K}}$ in the $C^{\ell+1}$ norm as $i\rightarrow \infty$, and we have
\begin{equation}\label{dr}
h_{\widetilde{K}}|D h_{\widetilde{K}}|^{q-(n+1)}\det(\nabla^{2}h_{\widetilde{K}}+h_{\widetilde{K}}I)=1,
\end{equation}
equipped with
\begin{equation}\label{dr2}
||h_{\widetilde{K}}-1||_{\infty}\geq \varepsilon_{0}, \quad \varepsilon_{0}>0.
\end{equation}
However,  by the uniqueness of even solutions to the isotropic dual Minkowski problem shown in Theorem \ref{ori} with $p=0$, we know that \eqref{dr} only admits a solution $h_{\widetilde{K}}\equiv 1$, which contradicts to \eqref{dr2}.
This completes the proof of Theorem \ref{coro2}.
\end{proof}

\section*{Acknowledgment} The author would like to thank  Yong Huang and Mohammad N. Ivaki for their interest in this work and  their valuable insights. The author also thanks the referees for their careful reading and  thorough comments.

{\bf Conflict of interest:} The author declares that there is no conflict of interest.

{\bf Data availability:} No data was used for the research described in the article.

\end{document}